\documentclass[a4paper,11pt]{article}

\usepackage{verbatim}

\usepackage[T1]{fontenc}
\usepackage[latin1]{inputenc}
\usepackage{setspace}
\usepackage{indentfirst}

\usepackage{geometry}
\usepackage{hyperref}

\usepackage{times}
\usepackage{amsthm}
\usepackage{amsmath}
\usepackage{amsfonts}
\usepackage{amssymb}
\usepackage{mathdots}
\usepackage{color}
\usepackage{mathrsfs}
\usepackage{stmaryrd}
\SetSymbolFont{stmry}{bold}{U}{stmry}{m}{n}

\usepackage{bm} 

\usepackage{tikz-cd}
\usepackage{tikz}
\usetikzlibrary{matrix,arrows,decorations.pathmorphing}
\usepackage[all]{xy}
\usepackage{graphicx}
\usepackage{caption}
\usepackage{tkz-graph}
\usepackage{subcaption}

\DeclareMathOperator{\diag}{diag}

\newtheorem{theorem}{Theorem}[section]
\newtheorem{proposition}[theorem]{Proposition}
\newtheorem{conjecture}[theorem]{Conjecture}
\newtheorem{lemma}[theorem]{Lemma}

\theoremstyle{definition}
\newtheorem{example}[theorem]{Example}
\newtheorem{remark}[theorem]{Remark}

\begin{document}
\title{The $U(n)$ Gelfand-Zeitlin system as a tropical limit of Ginzburg-Weinstein diffeomorphisms}

\author{Anton Alekseev\\ Anton.Alekseev@unige.ch \and Jeremy Lane\\ jeremy.lane@unige.ch \and Yanpeng Li\\ yanpeng.li@unige.ch}


\date{}

\footnotetext{\emph{Keywords:} integrable systems, Poisson-Lie groups, tropicalization, Poisson geometry, Gelfand-Zeitlin}

\maketitle

\begin{abstract}
	We show that the Ginzburg-Weinstein diffeomorphism $\mathfrak{u}(n)^* \to U(n)^*$ of \cite{AM} admits a scaling tropical limit on an open dense subset of $\mathfrak{u}(n)^*$. The target of the limit map is a product $\mathcal{C} \times T$, where $\mathcal{C}$ is the interior of a cone, $T$ is a torus, and $\mathcal{C} \times T$ carries an integrable system with natural action-angle coordinates. The pull-back of these coordinates to $\mathfrak{u}(n)^*$ recovers the Gelfand-Zeitlin integrable system of Guillemin-Sternberg \cite{GS}. As a by-product of our proof, we show that the Lagrangian tori of the Flaschka-Ratiu integrable system on the set of upper triangular matrices meet the set of totally positive matrices for sufficiently large action coordinates.
\end{abstract}
\section{Introduction}\label{section 1}

One of the richest settings for the study of integrable systems is the dual vector space $\mathfrak{k}^*$ of a finite dimensional Lie algebra $\mathfrak{k}$, equipped with its canonical Lie-Poisson structure. There are many important examples of integrable systems defined on $\mathfrak{k}^*$, include spinning tops \cite{A}, and Mishchenko-Fomenko systems \cite{MF,S}. Systems on $\mathfrak{k}^*$ also give rise to collective integrable systems via moment maps \cite{GS2,GS3}, for instance, leading to complete integrability of the geodesic flow on certain homogeneous spaces \cite{T}. 

Gelfand-Zeitlin systems, defined by Guillemin and Sternberg on the space of Hermitian matrices (interpreted as $\mathfrak{u}(n)^*$), are one of the most famous examples of such integrable systems \cite{GS}. 
Unlike Mishchenko-Fomenko systems, Gelfand-Zeitlin systems have natural global action-angle coordinates. This structure has lead to results about the symplectic topology of coadjoint orbits \cite{P,NNU}. Unfortunately, Gelfand-Zeitlin systems have only been defined for $\mathfrak{k}$ of type $A$, $B$ and $D$.

Motivated by the problem of generalizing Gelfand-Zeitlin systems, we are naturally brought to the following question: what underlying structures give rise to Gelfand-Zeitlin systems? One answer to this question comes from the study of toric degenerations, which define integrable systems with global action-angle coordinates on coadjoint orbits and many other spaces \cite{NNU,HK}. In this paper we make progress towards a new answer to this question by relating Gelfand-Zeitlin systems on $\mathfrak{u}(n)^*$ to Ginzburg-Weinstein diffeomorphisms and upper cluster algebra structures on dual Poisson-Lie groups.  We believe that this approach can be generalized and will give examples of new integrable systems on Lie algebra duals that have natural global action-angle coordinates.

In more detail, let $K$ be a compact connected Poisson-Lie group (that is, a Lie group equipped with a multiplicative Poisson bracket). Poisson-Lie theory associates to $K$ a dual Poisson-Lie group $K^*$ which is solvable. 
The Ginzburg-Weinstein Theorem says that $K^*$ admits global linearization maps: Poisson isomorphisms $\mathfrak{k}^* \to K^*$ called \emph{Ginzburg-Weinstein diffeomorphisms}. 

When $K=U(n)$, it is possible to describe explicit Ginzburg-Weinstein diffeomorphisms \cite{AM}. The dual Poisson-Lie group $U(n)^*$ has a completely integrable system with global action-angle coordinates due to Flaschka-Ratiu \cite{FR}. On open dense subsets of $\mathfrak{u}(n)^*$ and $U(n)^*$ a Poisson isomorphism is given by the identity in global action-angle coordinates for the Gelfand-Zeitlin and Flaschka-Ratiu systems (the main result of \cite{AM} is that this isomorphism extends to all of $\mathfrak{u}(n)^*$ and $U(n)^*$).

Returning to the general case of compact connected $K$, equip $\mathfrak{k}^*$ with its Lie-Poisson structure $\pi_{\mathfrak{k}^*}$, the dual Poisson-Lie group $K^*$ with its Poisson structure $\pi_{K^*}$, and fix a Ginzburg-Weinstein diffeomorphism ${\rm gw}\colon \mathfrak{k}^* \to K^*$. Since $\pi_{\mathfrak{k}^*}$ is linear, for positive $t$ the scaled Ginzburg-Weinstein diffeomorphism ${\rm gw}_t(A)={\rm gw}(tA)$ is a Poisson isomorphism with respect to $\pi_{\mathfrak{k}^*}$ and the scaled Poisson structure $t \pi_{K^*}$.

It was recently shown that integrable systems can be constructed from $K^*$ by tropicalizing its Poisson structure, which means taking the scaling limit $t\to\infty$ of $t\pi_{K^*}$ \cite{AD, ABHL}. More precisely, denoting $n = {\rm rank}(K)$ and $m = 1/2\left(\dim(K) - {\rm rank}(K)\right)$,
there exist $t$-dependent coordinate charts (coming from an upper cluster algebra structure on a double Bruhat cell)
\[
    \Delta_t \colon K^* \to \mathbb{R}^{n+m}\times T^m
\]
such that as $t\to \infty$ the Poisson structure $(\Delta_t)_*(t\pi_{K^*})$ converges to a constant Poisson structure $\pi_{\infty}$, of rank $m$, on $\mathcal{C} \times T^m$, where $\mathcal{C}\subseteq \mathbb{R}^{n+m}$ is the interior of a certain convex polyhedral cone and $T^m$ is a torus. In terms of coordinates $\zeta$ on $\mathbb{R}^{n+m}$ and $\varphi$ on $T^m$, the constant Poisson structure $\pi_{\infty}$ is of the form
\[
    \left\{ \zeta_i,\zeta_j\right\}_{\infty} = 0,\,\left\{ \varphi_i,\varphi_j\right\}_{\infty} = 0,\, \left\{ \zeta_i,\varphi_j\right\}_{\infty} = \pi_{i,j}
\]
where $\pi_{i,j}$ are some constants (see Theorem \ref{th: AD theorem} and the following discussion and \cite{AD} for more details in the $U(n)$ case. For the general case, see \cite{ABHL}). Since $\pi_{\infty}$ has rank $m$, the coordinates $\zeta$ are global action coordinates for an integrable system on $\mathcal{C} \times T^m$ (the coordinates $\varphi$ differ from global angle coordinates by a linear transformation). 

In general, we see that the composition 
\begin{equation}\label{finite t}
    \begin{tikzcd}
        \mathfrak{k}^*\arrow[r,"\text{gw}_t"] & [2.5em] {K^*} \arrow[r, "\Delta_t"] & [1.6em] {\mathbb{R}^{n+m}\times T^m}
    \end{tikzcd}
\end{equation}
is a Poisson isomorphism with respect to $\pi_{\mathfrak{k}^*}$ and $(\Delta_t)_*(t\pi_{K^*})$ for all positive finite $t$ (on the open dense subset where the composition is defined). One may then make the following conjecture.
\begin{conjecture}\label{conjecture}
    The limit as $t\to \infty$ of the map \eqref{finite t} exists on an open dense subset $\mathcal{U}\subseteq \mathfrak{k}^*$ and defines a Poisson isomorphism between $\mathcal{U}$, equipped with $\pi_{\mathfrak{k}^*}$, and $\mathcal{C}\times T^m$, equipped with $\pi_{\infty}$. 
\end{conjecture}
If this conjecture is true, then the limit defines a completely integrable system with global action-angle coordinates on an open dense subset of $\mathfrak{k}^*$.  Our main result is that Conjecture \ref{conjecture} holds for $K=U(n)$.

\begin{theorem}\label{main theorem}
	For the Ginzburg-Weinstein diffeomorphism of \cite{AM} and coordinate charts $\Delta_t$ defined by the cluster coordinates used in \cite{AD}, Conjecture \ref{conjecture} is true.  Moreover, the integrable system on the open dense subset $\mathcal{U} \subseteq \mathfrak{u}(n)^*$ obtained by pulling back action-angle coordinates from $\mathcal{C}\times T^m$ is the Gelfand-Zeitlin system (up to a linear transformation).
\end{theorem}

Theorem \ref{main theorem} is illustrated by the following example of $U(2)$ where the map ${\rm gw}_t$ can be made completely explicit. For $U(n)$, $n\geq 3$, the map ${\rm gw}_t$ does not admit a tractable closed formula in standard matrix coordinates.

\begin{example} If $\mathfrak{u}(2)^*$ is identified with Hermitian $2\times 2$ matrices, with coordinates
\[ 
    A = \left(\begin{array}{cc}
				x+y & z \\
				\overline{z} & x-y
		\end{array}\right) ,\, x,y\in \mathbb{R},\, z \in \mathbb{C} 
\]
and $U(2)^*$ is identified with upper triangular matrices, with positive diagonal entries, 
\[
    \left(\begin{array}{cc}
				a_1 & b \\
				0 & a_2
	\end{array}\right) ,\, a_1,a_2\in \mathbb{R}_+,\, b \in \mathbb{C},
\]
then the $t$-scaling of the Ginzburg-Weinstein diffeomorphism of \cite{AM} is given by the formula
\[
    {\rm gw}_t(A) = \left(\begin{array}{cc}
		    e^{t(x+y)/2} & e^{i\theta}\sqrt{e^{t(x+r)} + e^{t(x-r)} - e^{t(x+y)} - e^{t(x-y)}} \\
		    0 & e^{t(x-y)/2}
		   \end{array}\right)
\]
where $r = \sqrt{y^2 + |z|^2}$, $z = \rho e^{i\theta}$, and $t>0$.													
The action coordinates of the Gelfand-Zeitlin system are 
\[ 
    \lambda^{(2)}_1 = x+r,\, \lambda^{(2)}_2 = x-r,\, \lambda^{(1)}_1 = x-y.
\]
and they satisfy ``interlacing inequalities'' $\lambda^{(2)}_1\geq \lambda^{(1)}_1\geq \lambda^{(2)}_2$. The functions $\lambda^{(2)}_1,\lambda^{(2)}_2$ are Casimir functions and the function $\lambda^{(1)}_1$ generates a $S^1$ action on the coadjoint orbits (which are 2-spheres).

In coordinates $(\zeta_1^{(1)},\zeta_1^{(2)},\zeta_2^{(2)},\varphi_1^{(2)})$ on $\mathbb{R}^3 \times T^1$, the chart $\Delta_t$ is given by
\[
    \zeta_1^{(1)} = \frac{1}{t}\ln(a_2), \, \zeta_1^{(2)} = \frac{1}{t}\ln(|b|), \, \zeta_2^{(2)} = \frac{1}{t}\ln(a_1a_2),\, \varphi_1^{(2)} = {\rm Arg}(b).
\]
The open dense subset of $\mathcal{U}\subseteq \mathfrak{u}(n)^*$ is the set where the interlacing inequalities are strict. For $A$ in this subset, one computes using the interlacing inequalities that
\[
    \lim_{t\to \infty} \zeta_1^{(2)} \circ {\rm gw}_t = \lim_{t\to \infty}\frac{1}{2t}\ln\left(e^{t\lambda^{(2)}_1} + e^{t\lambda^{(2)}_2} - e^{t(\lambda^{(2)}_1+\lambda^{(2)}_2-\lambda^{(1)}_1)} - e^{t\lambda^{(1)}_1}\right) =\frac{1}{2}\lambda^{(2)}_1.
\]
while $\zeta_1^{(1)}\circ{\rm gw}_t = \lambda_1^{(1)}/2$ and $\zeta_2^{(2)}\circ{\rm gw}_t = (\lambda_1^{(2)}+\lambda_2^{(2)})/2 $ for all $t>0$. 
Thus,
\[
    \left(\begin{array}{c} \zeta_1^{(2)} \\\zeta_2^{(2)} \\ \zeta_1^{(1)} \end{array}\right) \circ {\rm gw}_{\infty} =  \frac{1}{2}\left(\begin{array}{ccc} 1 & 0 & 0 \\
                                        1 & 1 & 0 \\
                                        0 & 0 & 1 \\
    \end{array}\right)\left(\begin{array}{c} \lambda_1^{(2)} \\ \lambda_2^{(2)} \\ \lambda_1^{(1)} \end{array}\right).
\]
\end{example}

The organization of this paper is as follows. In Section \ref{section 2}, we recall the definition of action-angle coordinates for the Gelfand-Zeitlin system on $\mathfrak{u}(n)^*$, some background from the theory of Poisson-Lie groups, and the explicit formula for the Ginzburg-Weinstein diffeomorphisms of \cite{AM}.  Section \ref{section 3} describes the cluster coordinates on $U(n)^*$ that were used in \cite{AD} to tropicalize the Poisson bracket on $U(n)^*$, and explains their relation to the Gelfand-Zeitlin functions via the maps ${\rm gw}_t$. In Section \ref{section 4}, we introduce matrix factorization coordinates on $U(n)^*$ and the tropical estimation results from \cite{APS2}, which are the main ingredient for proving convergence.  Finally, Sections \ref{section 5} and \ref{section 6} are dedicated to the proof of Theorem \ref{main theorem}, which is divided into two propositions. Convergence of action coordinates is proven in Proposition \ref{proposition 5.1}, and convergence of angle coordinates is proven in Proposition \ref{proposition 6.1}. Part of the proof of Proposition \ref{proposition 6.1} involves showing that the Flaschka-Ratiu tori meet the set of totally positive matrices for sufficiently large values of action variables. A list of notation used throughout the paper is provided in Table \ref{table}.

\begin{table}[!h]
\caption{List of notation.}\label{table}
\begin{tabular}{ll}
$\mathcal{H}$          & The set of $n\times n$ Hermitian matrices\\
$\mathcal{H}^+$        & The set of positive definite $n\times n$ Hermitian matrices\\
$\mathcal{H}_0$        & The subset of $\mathcal{H}$ where all interlacing inequalities are strict\\
$\lambda_i^{(k)}$      & Gelfand-Zeitlin function on $\mathcal{H}$\\
$\psi_i^{(k)}$         & Angle coordinates for the Gelfand-Zeitlin system on $\mathcal{H}_0$\\
$\ell_i^{(k)}$         & The sum $\lambda_1^{(k)} + \cdots + \lambda_i^{(k)}$\\
$L$                    & The Gelfand-Zeitlin map with coordinates $\ell_i^{(k)}$ \\
$\Delta_{I,J}$         & Minor consisting of rows and columns $I,J\subseteq \{1,\ldots ,n\}$ respectively\\
$\Delta_i^{(k)}$       & Minor $\Delta_{I,J}$ with $I = \{ n-k+1,\ldots,n-k+i\}$ and $J = \{n-i,\ldots, n\}$  \\
$\zeta_i^{(k)},\varphi_i^{(k)}$       & Defined by $\Delta_i^{(k)} = e^{t\zeta_i^{(k)}+\sqrt{-1}\varphi_i^{(k)}}$   \\
$\Delta_t$             & The $t$-dependent coordinate chart defined by $\zeta_i^{(k)},\varphi_i^{(k)}$ \\
$m_i^{(k)}$            & Tropical Gelfand-Zeitlin functions obtained by tropicalizing the polynomials \eqref{eq: polynomials}\\
$L_{\mathbb{T}}$       & The tropical Gelfand-Zeitlin map with coordinates $m_i^{(k)}$\\
$\mathcal{L}_t$        & The $t$-dependent Gelfand-Zeitlin map on $AN$ \\
$\mathcal{C}$          & The interior of the cone defined by rhombus inequalities \eqref{eq: rhombus inequalities}\\
$\mathcal{C}^{\delta}$ & The subset of $\mathcal{C}$ defined by inequalities \eqref{eq: delta rhombus inequalities}\\
\end{tabular}
\vspace*{-4pt}
\end{table}

\section{Gelfand-Zeitlin and Ginzburg-Weinstein}\label{section 2}

In the first part of this section, we recall the definition of the classical Gelfand-Zeitlin system, along with several details about action-angle coordinates and the Gelfand-Zeitlin cone. In the second part of this section, we briefly recall the theory of Poisson-Lie groups, the Ginzburg-Weinstein theorem, and the explicit Ginzburg-Weinstein diffeomorphism of \cite{AM}.

\subsection{The classical Gelfand-Zeitlin system}\label{section 2.1}

For any Lie group $K$ with Lie algebra $\mathfrak{k}$, the dual vector space $\mathfrak{k}^*$ is endowed with a linear Poisson bracket called the \emph{Lie-Poisson} structure, defined by the formula
\[
    \left\{ f,g\right\}_{\mathfrak{k}^*}(\xi) = \langle \xi, [df_{\xi},dg_{\xi}]\rangle
\]
for $\xi \in \mathfrak{k}^*$ and smooth functions $f,g \in C^{\infty}(\mathfrak{k}^*)$. 

Let $K= U(n)$, the group of unitary $n\times n$ matrices, and let $\mathcal{H}$ denote the set of Hermitian $n\times n$ matrices.  We can identify $\mathfrak{u}(n)^* \cong \mathcal{H}$ via the non-degenerate bilinear form $(X,Y) = {\rm tr}(XY)$. With this identification, the \emph{Gelfand-Zeitlin functions} on $\mathfrak{u}(n)^*$ are the functions $\lambda_i^{(k)}\colon \mathcal{H} \to \mathbb{R}$, $1 \leq i \leq k \leq n$, defined so that for $A\in \mathcal{H}$, 
\[ \lambda^{(k)}_1(A) \geq \cdots \geq \lambda^{(k)}_k(A)\] 
	are the ordered eigenvalues of the $k\times k$ principal submatrix $A^{(k)}$ in the bottom-right corner of $A$ (see Figure \ref{fig 1}).
\begin{figure}[h]
\[
           \left(
           \begin{tikzpicture}
               [scale=0.2, baseline={(current bounding box.center)}]
              
                      \foreach \x in {0,...,8}{
                      \foreach \y in {0,...,-\x}
               \node at (\x,\y){ };
               }
 
               \draw (8,-2) -- (2,-2) -- (2,-8);
               \node at (4,-4) {$\ddots$};
               \draw (8,-6) -- (6,-6) -- (6,-8);
               \draw (8,-7) -- (7,-7) -- (7,-8);
           \end{tikzpicture}
           \right)
\]  
\caption{Bottom-right principal submatrices of $A$.}
\label{fig 1}
\end{figure}
	
The Gelfand-Zeitlin functions satisfy ``interlacing inequalities'',
\begin{equation}\label{eq: interlacing inequalities}
		 \lambda_i^{(k)} \geq \lambda_i^{(k-1)} \geq \lambda_{i+1}^{(k)}, \mbox{ for all }1 \leq i < k \leq n,
\end{equation}
	and the image of the \emph{Gelfand-Zeitlin map} $F\colon\mathcal{H}\to \mathbb{R}^{n(n+1)/2}$, defined by the Gelfand-Zeitlin functions, is the polyhedral cone defined by the inequalities \eqref{eq: interlacing inequalities}, called the \emph{Gelfand-Zeitlin cone}.
	
Let $\mathcal{H}_0$ denote the open dense subset of $\mathcal{H}$ where all the inequalities \eqref{eq: interlacing inequalities} are strict (this will be the set $\mathcal{U}$ in Theorem \ref{main theorem}).  The Gelfand-Zeitlin functions are smooth on $\mathcal{H}_0$ and define global action coordinates for a completely integrable system: the functions $\lambda_1^{(n)}, \ldots , \lambda_n^{(n)}$ are a complete set of Casimir functions, and the functions $\lambda_i^{(k)}$,  $1 \leq i \leq k < n$, generate 
\[
    m = n(n-1)/2
\]
commuting Hamiltonian $S^1$-actions, whose orbits coincide with the joint level-sets of the Gelfand-Zeitlin functions \cite{GS}. 

Angle coordinates on $\mathcal{H}_0$ corresponding to the global action coordinates $\lambda_i^{(k)}$ are defined by choosing a Lagrangian section $\sigma$ of the Gelfand-Zeitlin map and defining $\psi_i^{(k)}(p) = 0$ for all $p\in {\rm Im}(\sigma)$. The Gelfand-Zeitlin functions are invariant under the transpose map $A \mapsto A^T$, which is an anti-Poisson involution of $\mathcal{H}$. The fixed point set of the transpose map is the set ${\rm Sym}(n)$ of (real) symmetric $n\times n$ matrices and the intersection ${\rm Sym}_0(n) = {\rm Sym}(n)\cap \mathcal{H}_0$ is a union of Lagrangian submanifolds of $\mathcal{H}_0$ that are images of sections of the Gelfand-Zeitlin map. Thus, we may fix global angle coordinates for the Gelfand-Zeitlin system by choosing a connected component of ${\rm Sym}_0(n)$.  In these coordinates, the bivector of the Lie-Poisson bracket on $\mathfrak{u}(n)^*$ has the form 
\begin{equation}\label{eq: action angle bivector}
	\pi_{\mathfrak{k}^*} = \sum_{1\leq i \leq k < n} \frac{\partial}{\partial \lambda_i^{(k)}}\wedge\frac{\partial}{\partial \psi_i^{(k)}}.
\end{equation}
	
In this paper, it will be convenient to introduce the following notation. For all $1 \leq i \leq k \leq n$, let 
\begin{equation}\label{eq: definition of ell}
	\ell_i^{(k)} := \lambda_1^{(k)} + \cdots + \lambda_i^{(k)}.
\end{equation}
	The functions $\ell_i^{(k)}$ satisfy a list of inequalities equivalent to \eqref{eq: interlacing inequalities},
\begin{equation}\label{eq: rhombus inequalities}
	\ell_i^{(k+1)}+\ell^{(k)}_{i-1} \geq \ell_{i-1}^{(k+1)}+\ell^{(k)}_{i}, \mbox{ and } \ell_i^{(k+1)}+\ell^{(k)}_{i} \geq \ell_{i+1}^{(k+1)}+\ell^{(k)}_{i-1}.
\end{equation}
with the convention that $\ell_0^{(k)} = 0$ (these inequalities can be visualized as rhombi in a triangular tableau, see \cite[Figures 4 and 5]{APS1}). We denote the interior of the cone defined by \eqref{eq: rhombus inequalities} by $\mathcal{C}$. The set $\mathcal{C}$ is the image under a linear transformation of the interior of the Gelfand-Zeitlin cone. In coordinates $\ell_i^{(k)}, \psi_i^{(k)}$, the Poisson bivector \eqref{eq: action angle bivector} is no longer diagonal; it has a ``lower-triangular'' form,
\begin{equation}\label{eq: ell coordinates bivector}
	\pi_{\mathfrak{k}^*} = \sum_{1\leq i \leq k < n} \sum_{1\leq j \leq i} c_{i,j}^{(k)} \frac{\partial}{\partial \ell_i^{(k)}}\wedge\frac{\partial}{\partial \psi_j^{(k)}},
\end{equation}
where the coefficients $c_{i,j}^{(k)}$ are determined by \eqref{eq: action angle bivector} and \eqref{eq: definition of ell}.

\subsection{Poisson-Lie groups and Ginzberg-Weinstein diffeomorphisms}\label{section 2.3}


First introduced by Drinfel'd \cite{Dr} and Semenov-Tian-Shansky\cite{STS}, a \emph{Poisson-Lie group} is a Lie group $G$ equipped with a Poisson bivector $\pi$ such that group multiplication is a Poisson map. The linearization of $\pi$ at identity of $G$ is a $1$-cocycle on $\mathfrak{g}={\rm Lie}(G)$ with respect to the adjoint representation. This cocycle defines a Lie bracket on $\mathfrak{g}^*$ and endows the pair $(\mathfrak{g},\mathfrak{g}^*)$ with the structure of a Lie bialgebra. The dual Poisson-Lie group of $G$ is the connected, simply-connected Poisson-Lie group whose Lie bialgebra is $(\mathfrak{g}^*,\mathfrak{g})$. More details can be found in \cite{Lu,GW}.

Let $K$ be a connected compact Lie group. There are two natural  Poisson-Lie group structures on $K$, which in turn define two dual Poisson-Lie groups. We now explain these structures in more detail.

First, let $G = K^{\mathbb{C}}$, and fix an Iwasawa decomposition $G = KAN$, $\mathfrak{g} = \mathfrak{k} \oplus \mathfrak{a} \oplus \mathfrak{n}$, relative to a choice of maximal torus $T \subseteq K$ and positive roots. With these choices,  $\mathfrak{a} = \sqrt{-1}\mathfrak{t}$ and $\mathfrak{n}$ is the direct sum of positive root spaces. Let $B(\cdot,\cdot)$ be a nondegenerate, $K$-invariant, bilinear form on $\mathfrak{k}$, and let $B^{\mathbb{C}}$ be its complexification. The imaginary part of $B^{\mathbb{C}}$ defines a non-degenerate pairing between $\mathfrak{k}$ and $\mathfrak{a} \oplus \mathfrak{n}$ which identifies $\mathfrak{k}^* \cong \mathfrak{a} \oplus \mathfrak{n}$ and endows $\mathfrak{k}^*$ with the Lie algebra structure of $\mathfrak{a} \oplus \mathfrak{n}$. The pair $(\mathfrak{k},\mathfrak{k}^*)$ is a Lie bialgebra, which defines a Poisson-Lie group structure on $K$ such that the dual Poisson-Lie group $K^*$ is identified with $AN$.


Second, any Lie group has a trivial Poisson-Lie group structure when equipped with the zero Poisson bivector. The dual Poisson-Lie group of $(K,0)$ is $\mathfrak{k}^*$ equipped with the Lie-Poisson structure defined in the previous subsection. 

The linearization of the Poisson bivector of $K^*$ at the group unit equals the Poisson bivector of $\mathfrak{k}^*$. Therefore $K^*$ and $\mathfrak{k}^*$ are isomorphic as Poisson manifolds in a neighbourhood of their group units by local normal forms for Poisson manifolds \cite{Conn}. In fact, this isomorphism extends globally,

\begin{theorem}[Ginzburg-Weinstein Theorem]\cite{GW}
	The Poisson manifolds $(\mathfrak{k}^*,\pi_{\mathfrak{k}^*})$ and $(K^*,\pi_{K^*})$ are Poisson isomorphic.
\end{theorem}

Such Poisson isomorphisms are called \emph{Ginzburg-Weinstein isomorphisms/diffeomorphisms}. Note that $\mathfrak{k}^*$ is abelian, whereas $K^*$ is not, so Ginzburg-Weinstein diffeomorphisms can not be group homomorphisms, and $\mathfrak{k}^*$ and $K^*$ can not be isomorphic as Poisson-Lie groups. 

There are several proofs of the Ginzburg-Weinstein Theorem in the literature: the original proof \cite{GW} is an existence proof using a cohomology calculation, the proof in \cite{A} gives Ginzburg-Weinstein diffeomorphisms as flows of certain Moser vector fields, the proof in \cite{EEM} is by integration of a non-linear PDE of a classical dynamical $r$-matrix, and the proof in \cite{Boalch} uses the Stokes data of an ODE on a disc with an irregular singular point in the center.


For the remainder of this section, fix $K=U(n)$ and take the Iwasawa decomposition ${\rm GL}_n = KAN$ where $A$ is diagonal matrices with positive real entries and $N$ is upper triangular unipotent matrices. Let $\mathcal{H}^+$ denote the set of positive definite $n\times n$ Hermitian matrices. The map 
\begin{equation}\label{h}
    h\colon K^* \to \mathcal{H}^+,\, b\mapsto bb^*,
\end{equation}
is a diffeomorphism (with inverse given by Gaussian decomposition). It was observed by \cite{FR} that the functions $\ln(\lambda_i^{(k)}),$
define a completely integrable system on $\mathcal{H}^+$ (equipped with the Poisson structure $h_*\pi_{K^*}$). This system was related to the Gelfand-Zeitlin system on $\mathcal{H}$ by \cite{AM} who proved the following theorem.

\begin{theorem}\cite{AM}\label{theorem ginzburg-weinstein map}
	There is exists a Poisson isomorphism $\gamma\colon \mathcal{H} \to \mathcal{H}^+$ such that
	\begin{enumerate} 
		\item $\gamma$ intertwines the Gelfand-Zeitlin functions
		    \begin{equation}\label{eq: AM}
	            \lambda_i^{(k)}(A) = \ln\left(\lambda_i^{(k)}(\gamma(A))\right), \, \forall 1 \leq i \leq k \leq n.
            \end{equation}
		\item $\gamma$ intertwines the Gelfand-Zeitlin torus actions on $\mathcal{H}_0$ and $\mathcal{H}_0^+$.
		\item For any connected component $\mathcal{S}\subseteq Sym_0(n)$, $\gamma(\mathcal{S})\subseteq \mathcal{S}$.
		\item $\gamma$ is equivariant with respect to the conjugation action of $T\subseteq U(n)$.
		\item $\gamma(A + uI) = e^u\gamma(A)$.
		\item $\gamma(\overline{A}) = \overline{\gamma(A)}$.
	\end{enumerate}
\end{theorem}

\begin{remark}
    The map $\gamma$ is uniquely determined by properties (i)-(iii). The map $h^{-1}\circ\gamma$ is a Ginzburg-Weinstein diffeomorphism. One should note that \cite{AM} use a slightly different $h$, but  the composition $h^{-1}\circ\gamma$ remains Poisson. See Remark \ref{remark: choice of map h}.
\end{remark}

\section{Cluster coordinates on dual Poisson-Lie groups}\label{section 3} 

Let $K = U(n)$ and $K^* = AN$ as in the previous section. We define coordinates on an open dense subset of $AN$ following \cite{AD}: for all $1 \leq i \leq k \leq n$, let $\Delta_i^{(k)}$ denote the minor which is the determinant of the solid $i\times i$ submatrix formed by intersecting rows $n-k+1$ to $n-k+i$ and the last $i$ columns (see Figure \ref{fig 2}).
\begin{figure}[h]
\[
	 \left(
	 \begin{tikzpicture}
	    [scale=0.2, baseline={(current bounding box.center)}]
	    
		\foreach \x in {0,...,10}{
		\foreach \y in {0,...,-\x}
	    \node at (\x,\y){ };
	    }

	    \draw (5,-2) -- (10,-2) -- (10,-7) -- (5,-7) -- (5,-2);
	    \draw[dashed] (0,0)--(10,-10);
	    \node at (7.5,-4.5) {$i\times i$};
	\end{tikzpicture}
	\right)
	\begin{tikzpicture}
	    [scale=0.2, baseline={(current bounding box.center)}]
	    
		\foreach \x in {0,...,10}{
		\foreach \y in {0,...,-\x}
	    \node at (\x,\y){ };
	    }

	    \draw [dash dot] (0,-2)--(3,-2);
	    \draw [dash dot] (0,-7)--(3,-7);
	    \node at (7,-2) {$n-k+1$};
	    \node at (7,-7) {$n-k+i$};
	\end{tikzpicture}
\]
\caption{The minors $\Delta_i^{(k)}$.}
\label{fig 2}
\end{figure}
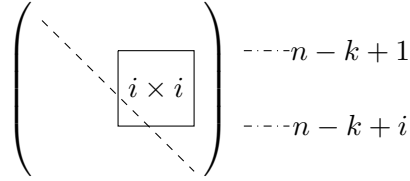

The minor $\Delta_i^{(k)}$ is a principal minor if and only if $i=k$, in which case it takes values in $\mathbb{R}_+$, otherwise it takes values in $\mathbb{C}$. Let $m=n(n-1)/2$; the number of such minors with $i<k$. Together, the $\Delta_i^{(k)}$'s  define a map 
\[
	\Delta\colon AN \to (\mathbb{R}_+)^n \times \mathbb{C}^{m},
\] 
whose restriction to the open dense subset where $\Delta_i^{(k)} \neq 0$ is a coordinate chart (diffeomorphism). On this subset, the equations 
\begin{equation*}
	\begin{split}
		e^{t\zeta_k^{(k)}} & = \Delta_k^{(k)},\, 1 \leq k \leq n ;\\
		e^{t\zeta_i^{(k)}+ \sqrt{-1}\varphi_i^{(k)}} & = \Delta_i^{(k)}, \, 1 \leq i < k \leq n
	\end{split}
\end{equation*}
define a $t$-dependent, polar coordinate chart
\[
	\Delta_t\colon AN \to (\mathbb{R}_+)^{n+m} \times T^{m},
\]
where $T=\mathbb{R}/2\pi\mathbb{Z}$, and $(\mathbb{R}_+)^{n+m} \times T^{m}$ is equipped with coordinates $(\zeta,\varphi)$ defined by the equations above. The diffeomorphisms $\Delta_t$ define Poisson structures
\[
	\pi_t := (\Delta_t)_* (t\pi_{K^*})
\]
on $(\mathbb{R}_+)^{n+m} \times T^{m}$.

Let $\mathcal{C}\subseteq \mathbb{R}^{n+m}$ be the interior of the cone defined by the inequalities \eqref{eq: rhombus inequalities}, replacing $\ell_i^{(k)}$ with $\zeta_i^{(k)}$, and for all $\delta >0$, define $\mathcal{C}^{\delta}$ as the subset where
\begin{equation}\label{eq: delta rhombus inequalities}
	\zeta_i^{(k+1)}+\zeta^{(k)}_{i-1} > \zeta_{i-1}^{(k+1)}+\zeta^{(k)}_{i} + \delta \mbox{ and } \zeta_i^{(k+1)}+\zeta^{(k)}_{i} > \zeta_{i+1}^{(k+1)}+\zeta^{(k)}_{i-1} + \delta.
\end{equation}

\begin{theorem}\cite{AD} \label{th: AD theorem}
    For $t>0$ and $(\zeta,\varphi) \in \mathcal{C}^{\delta}\times T^m$,		
\begin{equation}\label{eq: coefficients of pi infinity}
	\begin{split}
		\left\{ \zeta_i^{(k)}, \zeta_{q}^{(p)}\right\}_t     & = O(e^{-\delta t}).\\
		\left\{ \zeta_i^{(k)}, \varphi_{q}^{(p)}\right\}_t   & = \frac{1}{4}(\varepsilon(k-p)-1)(C-R) + O(e^{-\delta t}).\\
		\left\{ \varphi_i^{(k)}, \varphi_{q}^{(p)}\right\}_t & = O(e^{-\delta t}).\\
	\end{split}
\end{equation}
where $C$ is the number of columns that $\Delta_i^{(k)}$ and $\Delta_q^{(p)}$ have in common, $R$ is the number of rows that $\Delta_i^{(k)}$ and $\Delta_q^{(p)}$ have in common, and $\varepsilon(x) = x/|x|$ with $\varepsilon(0)=0$. Note that $\delta$ is constant in the big-O notation $O(e^{-\delta t})$. 
\end{theorem}

This theorem prompts the definition of a Poisson manifold $(\mathcal{C}\times T^m,\pi_{\infty})$ where is $\pi_{\infty}$ is the constant Poisson structure obtained by taking the $t\to \infty$ limit of $\pi_t$, with coefficients given by \eqref{eq: coefficients of pi infinity}. The coordinates $\zeta_1^{(n)},\ldots , \zeta_n^{(n)}$ are a complete set of Casimir functions for $\pi_{\infty}$. By \cite[Theorem 7]{AD}, there exists a Poisson isomorphism from $(\mathcal{C}\times T^m,\pi_{\infty})$ to $(\mathcal{H}_0,\pi_{\mathfrak{k}^*})$, linear with respect to coordinates $\ell,\psi$ as in Section \ref{section 2}, of the form
\begin{equation}\label{eq: AD poisson isomorphism}
	\left(\begin{array}{c} \zeta \\ \varphi \end{array}\right)
	= \left(\begin{array}{cc}
		\frac{1}{2}I & 0 \\
		0   & B \\
	\end{array}\right)
	\left(\begin{array}{c} \ell \\ \psi \end{array}\right),
\end{equation}
where $B$ is the integer matrix defining an automorphism of $T^m$, uniquely determined so that the map is Poisson.  Expanding this formula, the isomorphism is of the form
\begin{equation}\label{ordering of angles}
    \begin{split}
        \zeta_i^{(k)}   & = \frac{1}{2}\ell_i^{(k)} \\
        \varphi_i^{(k)} & = \psi_i^{(k-1)} + \mbox{ ``higher terms...''}
    \end{split}
\end{equation}
where $\psi_q^{(p)}$ is higher than $\psi_i^{(k)}$ if $p<k$ or $p=k$ and $q>i$.

\begin{remark}
	Theorem \ref{th: AD theorem} has been generalized in \cite{ABHL}. Given an arbitrary complex semisimple Lie group $G$ and a reduced word for the longest element $w_0$ of the Weyl group, the coordinate algebra $\mathbb{C}[G^{e,w_0}]$ of the double Bruhat cell $G^{e,w_0}=B\cap B_-w_0B_-$ carries an upper cluster algebra structure along with a cluster seed \cite{BFZ}. In the coordinates provided by the cluster seed, the Poisson structure on the dual Poisson-Lie group $K^*$ ($K^{\mathbb{C}} = G$) admits a ``partial tropicalization'': a $t$-deformation such that the $t=\infty$ limit is an integrable system $(\mathcal{C}\times T^m, \pi_{\infty})$, where $\mathcal{C}$ is the interior of an extended string-cone and $\pi_\infty$ is a constant Poisson bracket. Setting $G = {\rm GL}_n$ and taking the standard reduced word for $w_0$, the cluster seed is given by the minors $\Delta_i^{(k)}$ and one recovers Theorem \ref{th: AD theorem}.
\end{remark}  

Given a $n\times n$ matrix $b$, and subsets $I,J \subseteq \{1,\ldots,n\}$ with $|I| = |J|$, let $\Delta_{I,J}(b)$ denote the determinant of the  sub-matrix with rows $I$ and columns $J$. For any $1\leq i \leq n$, by the Cauchy-Binet formula,
\begin{equation}\label{eq: cauchy-binet}
    \sum_{|I|=i} \prod_{j\in I}\lambda_j(bb^*) 
    = \sum_{|I|=i}\Delta_{I,I}(bb^*)
    = \sum_{|I|=|J|=i}|\Delta_{I,J}(b)|^2,
\end{equation}
where $\lambda_j(bb^*)$ denote the eigenvalues of $bb^*$ (in our notation, these are the Gelfand-Zeitlin functions $\lambda_j^{(n)}$). 

We define a family of Ginzburg-Weinstein diffeomorphisms
\begin{equation}
    {\rm gw}_t \colon \mathcal{H}\to AN,\, {\rm gw}_t(A) = h^{-1}(\gamma(tA)),
\end{equation}
where $\gamma$ is the Ginzburg-Weinstein diffeomorphism of \cite{AM}, and $h$ is as defined in Section \ref{section 2} (see Equation \eqref{h} and Theorem \ref{theorem ginzburg-weinstein map}). For all $t>0$, 
\[
    ({\rm gw}_t)_*\pi_{\mathfrak{k}^*} = (h^{-1}\circ \gamma)_*(t\pi_{\mathfrak{k}^*}) = t\pi_{K^*}.
\]
Denote $b_t = {\rm gw}_t(A)$. Combining \eqref{eq: cauchy-binet} and \eqref{eq: AM}, for all $1\leq i \leq n$,
\begin{equation}\label{eq: master equation 1} 
    \sum_{|I|=i} e^{t \sum_{j\in I}\lambda_j(A)}  
    = \sum_{|I|=|J|=i}|\Delta_{I,J}(b_t)|^2.
\end{equation}
Note that if $b$ is upper triangular, the lower right $k\times k$ submatrix $(bb^*)^{(k)} = b^{(k)}(b^{(k)})^*$, so Equation \eqref{eq: master equation 1} remains valid when the eigenvalues $\lambda_j$ are replaced by the Gelfand-Zeitlin functions $\lambda_j^{(k)}$, and we take $I,J \subseteq \{n-k+1,\ldots,n\}$.

\begin{remark}\label{remark: choice of map h}
    If we take $h(b) = \sqrt{bb^*}$ instead, then Equation \eqref{eq: master equation 1} \emph{does not} hold for the Gelfand-Zeitlin functions $\lambda_j^{(k)}$, since in general the bottom right $k\times k$ submatrix of $\sqrt{bb^*}$ is not equal to $\sqrt{b^{(k)}(b^{(k)})^*}$.
\end{remark}
	
Using the generalized Pl\"ucker relations, every minor $\Delta_{I,J}$ can be written as a Laurent polynomial in the $\Delta_i^{(k)}$'s (this is a particular example of the {\em Laurent phenomenon} in the theory of cluster algebras, see \cite{FZ}). We can therefore write each $|\Delta_{I,J}|^2$ as a Laurent polynomial $P_{I,J}$ (in variables $\Delta_i^{(k)}$ and $\overline{\Delta}_i^{(k)}$)
\begin{equation}\label{laurent polys}
    \vert\Delta_{I,J}\vert^2 = P_{I,J}\left(\Delta_i^{(k)}, \overline{\Delta}_i^{(k)}\right)_{i,k}.
\end{equation}
Thus, equations \eqref{eq: master equation 1} can be written in the form 
\begin{equation}\label{eq: master equation 2} 
    \sum_{|I|=j, I \subseteq \{1,\ldots k\}} e^{t \sum_{i\in I}\lambda_i^{(k)}(A)} 
    =  \sum_{\substack{|I|=|J|=j\\ I,J\subseteq \{n-k+1,\ldots,n\}}} P_{I,J}\left(\Delta_i^{(k)}(b_t), \overline{\Delta}_i^{(k)}(b_t)\right).
\end{equation}

\begin{example}\label{example: poly} For $n=3$, Equation \eqref{eq: master equation 2} corresponding to $j=1$, $k=3$ is
\begin{equation*}
	\begin{split}
		e^{t\lambda_1^{(3)}}+ e^{t\lambda_2^{(3)}}+ e^{t\lambda_3^{(3)}}  &= |\Delta_{1,3}|^2 + |\Delta_{2,3}|^2 + |\Delta_{3,3}|^2 + |\Delta_{2,2}|^2 + |\Delta_{1,1}|^2 + |\Delta_{1,2}|^2   \\
		& = \vert\Delta_1^{(3)}\vert^2 + \vert \Delta_1^{(2)}\vert^2 + \vert\Delta_1^{(1)}\vert^2 + \frac{\vert \Delta_2^{(2)}\vert^2}{\vert \Delta_1^{(1)}\vert^2} + \frac{\vert \Delta_3^{(3)}\vert^2}{\vert \Delta_2^{(2)}\vert^2}\\
		& + \bigg\vert \frac{\Delta_1^{(1)}\Delta_2^{(3)} + \Delta_1^{(3)}\Delta_2^{(2)}}{\Delta_1^{(2)}\Delta_1^{(1)}} \bigg\vert^2 \\
		& = e^{2t\zeta_1^{(3)}}+ e^{2t\zeta_1^{(2)}}+ e^{2t\zeta_1^{(1)}} +e^{2t(\zeta_2^{(2)}- \zeta_1^{(1)})}+e^{2t(\zeta_3^{(3)}- \zeta_2^{(2)})}\\
		& + \left(e^{2t(\zeta_2^{(3)}-\zeta_1^{(2)})}+e^{2t(\zeta_1^{(3)}+\zeta_2^{(2)}-\zeta_1^{(2)}-\zeta_1^{(1)})}\right. \\
		& \left.-2e^{t(\zeta_2^{(3)}-\zeta_1^{(1)}+\zeta_1^{(3)}+\zeta_2^{(2)}-2\zeta_1^{(2)})} \cos(\varphi_2^{(3)}-\varphi_1^{(3)})\right).
	\end{split}
\end{equation*}
For $t$ large, since $\lambda_i^{(k)}$ satisfy the interlacing inequalities, the dominant term on the left side is $e^{t\lambda_1^{(3)}}$. We will see in Section \ref{section 4} that if $\zeta \in \mathcal{C}$, the dominant term on the right side is $e^{2t\zeta_1^{(3)}} = \vert\Delta_1^{(3)}\vert^2$.
\end{example}

\section{Matrix factorizations and tropical Gelfand-Zeitlin}\label{section 4}

In this section, we rephrase the results of \cite{APS1}, who use planar networks, in the language of matrix factorizations. 

\subsection{Matrix factorizations}

Let $E_{ij}$ denote the matrix with $(i,j)$ entry equal to 1 and all other entries 0. For a complex number $z$, denote $e_i(z):=\exp(zE_{i,i+1})$. Any word $\mathbf{i} = (i_1,\cdots ,i_k)$ in the alphabet $\{1, \ldots , n-1\}$ defines a map 
\[
	\bm{z}_{\mathbf{i}}\colon \mathbb{C}^{k} \to N, (z_1,\dots,z_k)\mapsto  e_{i_1}(z_1)\cdots e_{i_k}(z_k),
\]
where $N\subseteq {\rm GL}_n$ is the subgroup of upper triangular unipotent matrices. We also define 
\[
	\bm{a} \colon (\mathbb{R}_{+})^{n} \to A,\, (x_1,\dots,x_n) \mapsto \diag(x_1,\dots,x_n),
\]
where $A\subseteq {\rm GL}_n$ is the subgroup of diagonal matrices with positive real diagonal entries. 

Recall that the Weyl group of ${\rm GL}_n$ is $S_n$ generated by simple reflections $s_1,\dots,s_{n-1}$. The \emph{length} $l(w)$ of $w\in S_n$ is the smallest integer $k$ such that $w$ can be written as the product of $k$ simple reflections. A word $\mathbf{i}=(i_1,\dots,i_k)$ is \emph{reduced} if $l(w)=k$, where $w=s_{i_1}\cdots s_{i_k}$. Let $w_0$ be the longest element of $S_n$ with length $l(w_0)=m$. The \emph{standard} reduced word for $w_0$ is
\[
	\mathbf{i}_0=(1,\dots,n-1,1,\dots,n-2,\dots,1,2,1).
\]
For any reduced word $\mathbf{i}$ of $w_0$, the map, or \emph{matrix factorization},
\[
	\bm{a}\bm{z}_{\mathbf{i}}\colon (\mathbb{R}_{+})^{n}\times \mathbb{C}^{m} \to AN,\, \bm{a}\bm{z}_{\mathbf{i}}(x,z) = \bm{a}(x_1, \ldots, x_n)\bm{z}_{\mathbf{i}}(z_1, \ldots, z_m)
\]
defines a chart, when restricted to $(\mathbb{R}_{+})^{n}\times(\mathbb{C}^{\times})^{m}$. Here we write $(x,z)$ as shorthand for the tuple $(x_1, \ldots , x_n,z_1, \ldots , z_m)$. When $\mathbf{i}$ is the standard the word $\mathbf{i}_0$, we write $\bm{a}\bm{z}_0$ for simplicity.

\begin{example}\label{GL3}
	For ${\rm GL}_3$ and the standard reduced word $\mathbf{i}_0=(1,2,1)$, we have:
	\begin{align*}
		\bm{a}\bm{z}_{0}(x,z) & =
		\begin{pmatrix}
			x_1 & 0 & 0 \\
			0 & x_2 & 0 \\
			0 & 0 & x_3
		\end{pmatrix}
		\begin{pmatrix}
			1 & z_1 & 0 \\
			0 & 1 & 0 \\
			0 & 0 & 1
		\end{pmatrix}
		\begin{pmatrix}
			1 & 0 & 0 \\
			0 & 1 & z_2  \\
			0 & 0 & 1
		\end{pmatrix}
		\begin{pmatrix}
			1 & z_3 & 0 \\
			0 & 1 & 0 \\
			0 & 0 & 1
		\end{pmatrix}\\
		&=
		\begin{pmatrix}
			x_1 & x_1(z_1+z_3) & x_1z_1z_2 \\
			0 & x_2 & x_2z_2 \\
			0 & 0  & x_3
		\end{pmatrix}.
	\end{align*}
\end{example}  

\begin{remark} More generally, if $G$ is a reductive group with a choice of positive roots, and $\mathbf{i} = (i_1, \cdots ,i_m)$ is a reduced word for the longest element $w_0$ of the Weyl group, one can define the maps $e_i(z)$ using a Chevalley basis, and a chart $\bm{a}\bm{z}_{\mathbf{i}}$ as above. See \cite{BZ}.
\end{remark}

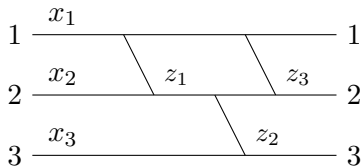
\begin{figure}[h]
\[
	\begin{tikzpicture}[baseline={([yshift=-.5ex]current bounding box.center)}, scale=1.6]
		\draw (-0.5,2)node[left]{$3$}--(2,2)node[right]{$3$};
		\draw (-0.5,2.5)node[left]{$2$}--(2,2.5)node[right]{$2$};
		\draw (-0.5,3)node[left]{$1$}--(2,3)node[right]{$1$};

		\draw (0.25,3)--(0.5,2.5) node[above right, font=\small]{$z_1$};
		\draw (1.25,3)--(1.5,2.5) node[above right, font=\small]{$z_3$};
		\draw (1,2.5)--(1.25,2) node[above right, font=\small]{$z_2$};

		\node at (-0.25, 3.15) {$x_1$};
		\node at (-0.25, 2.65) {$x_2$};
		\node at (-0.25, 2.15) {$x_3$};
	\end{tikzpicture}
\]  
\caption{The planar network representing the matrix factorization in Example \ref{GL3}.}
\label{fig 3}
\end{figure}

Matrix factorization coordinates on $AN\subseteq {\rm GL}_n$ can be represented by planar networks; weighted planar graphs, oriented left to right, with $n$ ``source'' vertices on the left, and $n$ ``sink'' vertices on the right (see \cite{APS1} for more details).  The matrix factorization  $\bm{az}_0$ is represented by the \emph{standard planar network}, $\Gamma$, with $n$ horizontal edges connecting the sources to the sinks (both labelled by $1,\ldots ,n$ as in Figure \ref{fig 3}), and $n(n-1)/2$ non-horizontal edges, arranged as in Figure \ref{fig 3} for the case $n=3$.  The non-horizontal edges (which correspond to $e_i(z)$'s in the matrix factorization) are labelled with the $z$'s and the horizontal edges are labelled at the sources with the $x$'s, as in Figure \ref{fig 3}.  The $(i,j)$ entry of $\bm{az}_0(x,z)$ equals
\[
    \sum_{\gamma \in P\Gamma(i,j)}\prod_{e\in \gamma} w(e), 
\]
where $P\Gamma(i,j)$ is the set of directed paths in $\Gamma$ from source $i$ to sink $j$, $e$ is an edge contained in $\gamma$, and $w(e)$ is the weight assigned to $e$ (weights not written on $\Gamma$ are the multiplicative identity).

The planar network representation of a matrix factorization gives minors of $\bm{az}_{\mathbf{i}}(x,z)$ a combinatorial interpretation. Recall from the previous section that $\Delta_{I,J}$ denotes the minor with rows $I$ and columns $J$, $|I| = |J| = i$. By the Lindstr\"om Lemma \cite{AD,FZ2},  
\[
    \Delta_{I,J}(\bm{az}_{\mathbf{i}}(x,z)) = \sum_{\gamma\in P\Gamma(I,J)}\prod_{e\in \gamma}w(e),
\]
where $P\Gamma(I,J)$ is the set of $i$-\emph{multipaths} from $I$ to $J$; a union of $i$ disjoint directed paths with sources in $I$ and sinks in $J$. For the minors $\Delta_i^{(k)}$ there is exactly one such multipath (see Figure \ref{fig 4}). 

\begin{example}\label{example: minors} Continuing from the previous example,  
\[ 
	\Delta_2^{(3)}(\bm{a}\bm{z}_{0}(x,z)) = x_1x_2z_1z_3,\, \Delta_{1,2}(\bm{a}\bm{z}_{0}(x,z)) = x_1z_1 + x_1z_3. 
\]

\end{example} 

As illustrated by the example,

\begin{theorem}\label{theorem: minors are positive}\cite[Theorem 5.8]{BZ}
    For any reduced word $\mathbf{i}$,  $\Delta_{I,J}(\bm{a}\bm{z}_{\mathbf{i}}(x,z))$ is a polynomial in the coordinates $x,z$ with positive coefficients. Moreover, the minors $\Delta_{i}^{(k)}(\bm{a}\bm{z}_{\mathbf{i}}(x,z))$ are all monomials.
\end{theorem}

One should note that \cite[Theorem 5.8]{BZ} is considerably more general.


\begin{figure}[h]
\centering
\begin{minipage}[b]{.3\linewidth}
\centering
\[
	\begin{tikzpicture}[baseline={([yshift=-.5ex]current bounding box.center)}, scale=1.6]
		\draw (0,2)--(2,2);
		\draw (0,2.5)--(2,2.5);
		\draw (0,3)--(2,3);

		\draw (0.25,3)--(0.5,2.5);

		\draw[line width=.5mm] (0,3)--(1.25,3)--(1.5,2.5)--(2,2.5);
		\draw[line width=.5mm] (0,2.5)--(1,2.5)--(1.25,2)--(2,2);
		\node at (0,1.9) {};
	\end{tikzpicture}
\]
\subcaption{The minor $\Delta_2^{(3)}$}
\end{minipage}
\begin{minipage}[b]{.6\linewidth}
\centering
\[
	\begin{tikzpicture}[baseline={([yshift=-.5ex]current bounding box.center)}, scale=1.6]
		\draw (0,2)--(2,2);
		\draw (0,2.5)--(2,2.5);
		\draw (0,3)--(2,3);

		\draw (1.25,3)--(1.5,2.5);
		\draw (1,2.5)--(1.25,2);

		\draw[line width=.5mm] (0,3)--(0.25,3)--(0.5,2.5)--(2,2.5);
		\node at (0,1.85) {};
	\end{tikzpicture}\quad +\quad
	\begin{tikzpicture}[baseline={([yshift=-.5ex]current bounding box.center)}, scale=1.6]
		\draw (0,2)--(2,2);
		\draw (0,2.5)--(2,2.5);
		\draw (0,3)--(2,3);
		
		\draw (0.25,3)--(0.5,2.5);
		\draw (1,2.5)--(1.25,2);

		\draw[line width=.5mm] (0,3)--(1.25,3)--(1.5,2.5)--(2,2.5);
		
		\node at (0,1.85){};
	\end{tikzpicture}
\]
\subcaption{The minor $\Delta_{1,2}^{}$}
\end{minipage}
\caption{Multipaths appearing in the minors of Example \ref{example: minors}.}
\label{fig 4}
\end{figure}
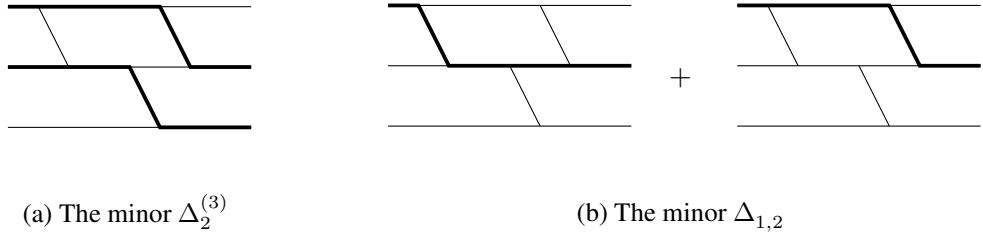

\subsection{Tropical Gelfand-Zeitlin functions}

A polynomial $p \in \mathbb{C}[x_1,\ldots,x_n]$ is \emph{positive} if all the coefficients are positive.  The \emph{tropicalization} of a positive polynomial $p(x) = \sum_{\bm{i}} c_{\bm{i}}x_1^{\alpha_{i_1}}\cdots x_n^{\alpha_{i_n}}$ is the piecewise linear function 
\[
	p^{\mathbb{T}}(w) := \max_{\bm{i}}\left\{ w_1\alpha_{i_1} + \cdots + w_n\alpha_{i_n}\right\}.
\]
If the variables $x$ are all real then, substituting $x_i = e^{tw_i}$, this equals the limit
\[
    \lim_{t\to \infty} \frac{1}{t}\ln\left(p(e^{tw})  \right).
\]
If the variables  are complex, then we may still define $p^{\mathbb{T}}$ as above, and -- on an open dense set -- this equals the function obtained by substituting $x = e^{tw+\sqrt{-1}\phi}$ and evaluating the limit as before (the limit does not equal $p^{\mathbb{T}}$ on the subset where leading terms cancel due to their complex arguments).

By Theorem \ref{theorem: minors are positive}, the minors $\Delta_{I,J}$ are positive polynomials in matrix factorization coordinates $\bm{a}\bm{z}_{\mathbf{i}}$. Thus, for all $1\leq i\leq k \leq n$, the polynomials 
\begin{equation}\label{eq: polynomials}
	p_i^{(k)}(x,z) = \sum_{\substack{|I|=|J| = i\\ I,J \subseteq \{ n-k+1,\ldots n\}}} |\Delta_{I,J}(\bm{a}\bm{z}_{\mathbf{i}}(x,z))|^2
\end{equation}
(which are simply the right sides of equations \eqref{eq: master equation 1}) are positive and can be tropicalized.  Since the $p_i^{(k)}(x,z)$ are related, by Equation \eqref{eq: master equation 1}, to the logarithmic Gelfand-Zeitlin functions on $AN$, their tropicalizations are called the \emph{tropical Gelfand-Zeitlin functions} and denoted $m_i^{(k)}$. One should note that our definition of $m_i^{(k)}$ differs from the definition of $m_i^{(k)}$ in \cite{APS2} by a factor of $2$.

\begin{theorem}\cite[Theorem 2]{APS1} 
For any choice of matrix factorization coordinates $\bm{a}\bm{z}_{\bf{i}}$, the tropical Gelfand-Zeitlin functions satisfy the rhombus inequalities
\begin{equation}\label{eq: tropical rhombus inequalities}
	m_i^{(k+1)}+m^{(k)}_{i-1} \geq m_{i-1}^{(k+1)}+m^{(k)}_{i}, \mbox{ and } m_i^{(k+1)}+m^{(k)}_{i} \geq m_{i+1}^{(k+1)}+m^{(k)}_{i-1}.
\end{equation}
\end{theorem}

Taken together, the tropical Gelfand-Zeitlin functions define a piecewise linear map, called the \emph{tropical Gelfand-Zeitlin map},
\[
    L_{\mathbb{T}}\colon \mathbb{R}^{n+m} \to \mathbb{R}^{n+m},\, L_{\mathbb{T}}(w) = (m_i^{(k)}(w))_{i,k}.
\]

\begin{theorem}\cite[Theorem 3]{APS1}
    If $L_{\mathbb{T}}$ is defined using the standard reduced word $\mathbf{i}_0$, then ${\rm Im}(L_{\mathbb{T}}) = \overline{\mathcal{C}}$ (the closure of $\mathcal{C}$).   
\end{theorem}

As $L_{\mathbb{T}}$ is piecewise linear, $\mathbb{R}^{n+m}$ decomposes into polyhedral chambers where $L_{\mathbb{T}}$ is linear.  There is a unique chamber, which we denote by $\mathcal{C}_{\mathbb{T}}$, where the rank of $L_{\mathbb{T}}$ is $n+m$, and for this chamber, $L_{\mathbb{T}}(\mathcal{C}_{\mathbb{T}}) = \mathcal{C}$ \cite{APS1}. For all $\delta>0$, \cite{APS2} define $W^\delta$ to be the set of all $w\in \mathbb{R}^{n+m}$ such that,
\[
	m_i^{(k+1)}+m^{(k)}_{i-1} > m_{i-1}^{(k+1)}+m^{(k)}_{i} + \delta, \mbox{ and } m_i^{(k+1)}+m^{(k)}_{i} > m_{i+1}^{(k+1)}+m^{(k)}_{i-1} + \delta,
\]
and for any two disjoint subsets $\alpha,\beta \subseteq E\Gamma$ (the edge set of the standard planar network $\Gamma$ defined in Section \ref{section 3}),
\[
    |w(\alpha)-w(\beta)|>\delta.
\]
Let $\bm{a}\bm{z}_{\mathbf{i}}(tw,\phi)$ denote $\bm{a}\bm{z}_{\mathbf{i}}(x,z)$ with the substitutions $x_i = e^{tw_i}$, $z_j = e^{tw_j+\sqrt{-1}\phi_j}$. The condition $|w(\alpha)-w(\beta)|>\delta$ guarantees that there is only one leading term in the polynomials $p_i^{(k)}(\bm{a}\bm{z}_{\mathbf{i}}(tw,\phi))$, which means that the tropicalization $m_i^{(k)}$ is equal to the tropical limit described at the beginning of the section. This is the content of the following proposition.

\begin{proposition}\cite[Proposition 2]{APS1}\label{proposition 4.1}
    Fix $\delta>0$ and let $w \in W^{\delta}$. For any reduced word  $\mathbf{i}$ of $w_0$, there is a constant $C$ such that for all $t\geq 1$ and any $\phi \in T^m$, 
    \[
    	\left\vert 	m_i^{(k)}(w) - \frac{1}{t}\ln\left(\prod_{j=1}^i\lambda_{j}^{(k)}\Big(\bm{a}\bm{z}_{\mathbf{i}}(tw,\phi) \bm{a}\bm{z}_{\mathbf{i}}(tw,\phi)^*\Big)\right) \right\vert < Ce^{-t\delta},
    \]
    for all $1\leq i\leq k\leq n$.
\end{proposition}

Let $\mathcal{L}_t\colon AN \to \mathbb{R}^{m+n}$ denote the map with coordinates given by 
\[
	\frac{1}{t}\ln\left( \prod_{j=1}^i \lambda_j^{(k)}(bb^*)\right).
\]

\begin{proposition}\cite[Proposition 5]{APS1}\label{proposition 4.2}
	For every $\delta > 0$, there exists $t_0>0$ such that for all $t \geq t_0$, the following statement holds: for all $b \in \mathcal{L}_t^{-1}(\ell),$ $\ell \in \mathcal{C}^{\delta},$ there exists $w\in W^{\delta/2}$ and $\phi \in T^m$ such that 
	\[
	    b = \bm{a}\bm{z}_{0}(tw,\phi).
	\]
\end{proposition}

We end this section with one last detail from \cite{APS1} that is crucial in the proof of Proposition \ref{proposition 5.1}. For all $w \in \mathcal{C}_{\mathbb{T}}$, the maximum in the definition of $m_i^{(k)}(w)$ is the sum corresponding to the monomial $\Delta_i^{(k)}(\bm{a}\bm{z}_0(w,\phi))$ (see also Appendix C of \cite{AD}). Thus, for all $w \in \mathcal{C}_{\mathbb{T}}$,
\begin{equation}\label{eq: lindstrom minors}
    m_i^{(k)}(w) = \frac{1}{t}\ln\vert \Delta_i^{(k)}(\bm{a}\bm{z}_0(tw,\phi))\vert^2.
\end{equation}

\section{Convergence of action coordinates}\label{section 5}

In this section we prove,
\begin{proposition}\label{proposition 5.1}
For all $\delta > 0$ and $1\leq i \leq k \leq n,$ there exists a constant $C$ and $t_0\geq 0$ such that $t\geq t_0$ implies 
	\begin{equation}	
		\bigg\vert 2\, \zeta_i^{(k)}\circ \Delta_t \circ {\rm gw}_t(A)- \ell_i^{(k)}(A)\bigg\vert < Ce^{-t\delta/2}
	\end{equation}
	for all $A \in \mathcal{H}^{\delta}_0(n)$.
\end{proposition}

\begin{proof} Let $\delta>0$, and fix $A\in \mathcal{H}^{\delta}_0(n)$. Denote ${\rm gw}_t(A) = b_t$. By definition of $\zeta_i^{(k)}$ and Theorem \ref{theorem ginzburg-weinstein map},
\begin{equation}\label{secondary limit}	
	\begin{split}
		\bigg\vert 2\,\zeta_i^{(k)}\circ \Delta_t \circ {\rm gw}_t(A)- \ell_i^{(k)}(A)\bigg\vert & = \bigg\vert \frac{2}{t} \ln\vert\Delta_i^{(k)}(b_t)\vert - \ell_i^{(k)}(A)\bigg\vert \\
		& = \bigg\vert \frac{1}{t} \ln\vert\Delta_i^{(k)}(b_t)\vert^2 - \frac{1}{t}\ln\left(\prod_{j=1}^i\lambda_{j}^{(k)}(b_tb_t^*)\right)\bigg\vert 
	\end{split}
\end{equation}
for all $t>0$.

Since $A \in \mathcal{H}^{\delta}_0(n)$, by Proposition \ref{proposition 4.2}, there exists $t_0>0$ such that for all $t \geq t_0$, there exists $w\in W^{\delta/2}$ and $\phi \in T^m$ such that $b_t = \bm{a}\bm{z}_{0}(tw,\phi)$.  
Since $w \in W^{\delta/2}$,  we can combine Equations \eqref{secondary limit}, \eqref{eq: lindstrom minors} and Proposition \ref{proposition 4.1} to conclude that there exists $C\geq0$ such that 
\begin{equation*}
\bigg\vert \frac{1}{t} \ln\vert\Delta_i^{(k)}(b_t)\vert^2 - \frac{1}{t}\ln\left(\prod_{j=1}^i\lambda_{j}^{(k)}(b_tb_t^*)\right) \bigg\vert = \bigg\vert 	m_i^{(k)}(w) - \frac{1}{t}\ln\left(\prod_{j=1}^i\lambda_{j}^{(k)}(b_tb_t^*)\right) \bigg\vert < Ce^{-t\delta/2},
\end{equation*}
 which completes the proof.
\end{proof}

%
\section{Convergence of angle coordinates}\label{section 6}

Recall from Section \ref{section 2} that a choice of connected component of ${\rm Sym}_0(n) := {\rm Sym}(n)\cap \mathcal{H}_0$ determines a choice of angle coordinates $\psi$ for the Gelfand-Zeitlin system. Recall also that (up to a linear transformation), the functions $\varphi$ are angle coordinates on $\mathcal{C}\times T^m$ with respect to $\pi_{\infty}$. In this section, we prove there exists a choice of angle coordinates $\psi$ for the Gelfand-Zeitlin system so that,

\begin{proposition}\label{proposition 6.1}
    On $\mathcal{H}_0$, for all $1 \leq i < k \leq n$,
    \[\lim_{t\to \infty} \varphi^{(k)}_i\circ \Delta_t \circ {\rm gw}_t = \psi^{(k-1)}_i + \mbox{ higher terms}.\]
    The sum on the right is a linear combination of the coordinates $\psi$ as in Equations \eqref{eq: AD poisson isomorphism} and \eqref{ordering of angles} (in particular, ``higher terms'' refers to the ordering of angle coordinates defined after Equation \eqref{ordering of angles}).
\end{proposition}

In subsection \ref{ss1}, we show that the Hamiltonian vector fields of the Flaschka-Ratiu system on $AN$ converge as $t\to \infty$ to the Hamiltonian vector fields of the action coordinates $\zeta$. In subsection \ref{ss2}, we prove that for $t$ sufficiently large, the fibers of the Flaschka-Ratiu system intersect the chamber $AN_+$ of matrices in $AN$ such that the minors $\Delta_i^{(k)}>0$ (one may think of this as the chamber of ``totally positive'' matrices in $AN$). These facts are combined in subsection \ref{ss3} to prove Proposition \ref{proposition 6.1} and Theorem \ref{main theorem}.


\subsection{Convergence of Hamiltonian vector fields}\label{ss1}

For $I\subseteq \{1,\ldots, k\}$, let $L_I$ be the linear function so that $L_I(\ell_1^{(k)},\ldots , \ell_k^{(k)}) = \sum_{i\in I}\lambda_i^{(k)}$ (cf. Equation \eqref{eq: definition of ell}). Recall also  that for $I,J \subseteq \{n-k+1,\ldots,n\}$, $|I|=|J|=j$, there exists a Laurent polynomial $P_{I,J}$ so that 
\[
	|\Delta_{I,J}|^2 = P_{I,J}\left(\Delta_i^{(k)},\overline{\Delta}_i^{(k)}\right)_{i,k}
\]
(see Equation \eqref{laurent polys}). For all $1\leq j \leq k \leq n$, rearrange Equation \eqref{eq: master equation 2} to define functions (for fixed finite $t$)
\begin{equation}\label{eq: implicit function theorem functions}
    f_j^{(k)}(\ell,\zeta,\varphi) = \sum_{\substack{|I|=j\\ I \subseteq \{1,\ldots k\}}} e^{t L_I(\ell)} 
    - \sum_{\substack{|I|=|J|=j\\ I,J\subseteq \{n-k+1,\ldots,n\}}} P_{I,J}\left(\zeta,\varphi\right),
\end{equation}
where we have substituted $\Delta_i^{(k)} = e^{t\zeta_i^{(k)}+\sqrt{-1}\varphi_i^{(k)}}$ into each $P_{I,J}$ (see Example \ref{example: poly}).

By Theorem \ref{theorem ginzburg-weinstein map}, the functions $\ell(\zeta,\varphi) = \mathcal{L}_t(\Delta_t^{-1}(\zeta,\varphi))$ solve the system of equations
\[
    f_j^{(k)}(\ell,\zeta,\varphi) = 0,\, 1 \leq j \leq k \leq n.
\]
The first step in the proof of Proposition \ref{proposition 6.1} is to apply the Implicit Function Theorem to this system of equations
to get asymptotic control on the partial derivatives of $\ell$ with respect to $\zeta,\varphi$.  

For both sums on the right side of Equation \eqref{eq: implicit function theorem functions}, there is a single term that dominates exponentially for large $t$. In the first sum, recall that if $\ell \in \mathcal{C}^{\delta}$, and $I \neq \{1,\ldots, i\}$, 
\begin{equation}\label{eq: estimate 1}
	L_I(\ell^{(k)}) - \ell_j^{(k)} = \sum_{i\in I} \lambda_i^{(k)} - \left(\lambda_1^{(k)}+\cdots + \lambda_j^{(k)}\right) < -\delta,
\end{equation}
so the term $e^{t\ell_j^{(k)}}$ dominates for large $t$. In the second sum, if $\zeta \in \mathcal{C}^{\delta}$, and $t>0$ is sufficiently large, then by Proposition \ref{proposition 4.2}, there exists $w\in W^{\delta}$ and $\phi\in T^m$ such that $\Delta_t(\bm{az}_0(tw,\phi)) = (\zeta,\varphi)$. Unpacking the proof of \cite[Proposition 2]{APS2}, for $I,J\subseteq \{n-k+1,\ldots,n\}$, $|I| = |J| = j$, not both equal to  $\{n-k+1, \ldots ,n-k+i\}$, 
\begin{equation}\label{eq: estimate 2}
	 e^{-2t\zeta_j^{(k)}}P_{I,J}\left(\zeta,\varphi\right) = \frac{\vert \Delta_{I,J}\vert^2}{\vert \Delta_j^{(k)}\vert^2} < Ce^{-t\delta},
\end{equation}
so the term $\vert\Delta_j^{(k)}\vert^2 = e^{2t\zeta_j^{(k)}}$ dominates the sum for large $t$. In other words, the $j\times j$ minors of the bottom right $k\times k$ submatrix are dominated exponentially by the corner minor $\Delta_j^{(k)}$. The reader may find it useful to work this out for a couple examples, or compare with Example \ref{example: poly}.

In what follows, we order the coordinates 
\[ 
\ell_1^{(n)},\ldots,\ell_n^{(n)},\ell_1^{(n-1)},\ldots,\ell_{n-1}^{(n-1)},\ldots, \ell_1^{(1)}
\]
and similarly for $\zeta$ and $\varphi$.

\begin{lemma}\label{lemma: convergence of derivatives}
	For $t>0$ sufficiently large, 
	\[
		\frac{\partial \ell_i^{(k)}}{\partial \zeta_q^{(p)}} = 2\delta_{k,p}\delta_{i,q}e^{t(2\zeta_q^{(p)}-\ell_i^{(k)})} + O(e^{-t\delta}),\,
		\frac{\partial \ell_i^{(k)}}{\partial \varphi_q^{(p)}} = O(e^{-t\delta}),
	\]
	where $\delta_{k,p}$ and $\delta_{i,q}$ are Kronecker delta functions.
\end{lemma}

Note that since the functions $f_j^{(k)}$ do not involve $\zeta^{(p)}$'s and $\varphi^{(p)}$'s for $p>k$, the partial derivatives
\[
	\frac{\partial \ell_j^{(k)}}{\partial \zeta_q^{(p)}},\, \frac{\partial \ell_j^{(k)}}{\partial \varphi_q^{(p)}} = 0, \mbox{ for all } p>k.
\]

\begin{proof} For finite $t$, let $f\colon \mathbb{R}^{n+m}\times \mathbb{R}^{n+m} \times T^m \to \mathbb{R}^{n+m}$ with coordinates
\[
    f_1^{(n)},\ldots,f_n^{(n)},f_1^{(n-1)},\ldots,f_{n-1}^{(n-1)},\ldots, f_1^{(1)}
\]
as defined in Equation \eqref{eq: implicit function theorem functions}. We will apply the Implicit Function Theorem at solution of $f(\ell,\zeta,\varphi)=0$ with $(\zeta,\varphi) \in \mathcal{C}^{\delta}\times T^m$ fixed, and 
\[
	\ell = \mathcal{L}_t(\Delta_t^{-1}(\zeta,\varphi)).
\]
Note that if $\zeta \in \mathcal{C}^{\delta}$, then by Proposition \ref{proposition 5.1}, $\ell \in \mathcal{C}^{\delta/2}$ for $t$ sufficiently large.

Order the coordinates $\ell,\zeta,\varphi$ as above.  The matrix of partial derivatives of $f$ with respect to $\ell$ is the $(n+m)\times(n+m)$ block diagonal matrix
	\begin{equation*}
		D_{\ell}f = \left(\frac{df_i^{(k)}}{d\ell_j^{(p)}}\right) = \left( \begin{array}{ccc}
			D_{\ell^{(n)}}f^{(n)} & & \\
			& \ddots &       \\
			&& D_{\ell^{(1)}}f^{(1)}
		\end{array}\right),
	\end{equation*}
	where $f^{(k)} = (f^{(k)}_1,\ldots , f^{(k)}_k)$ and $\ell^{(k)} = (\ell^{(k)}_1,\ldots , \ell^{(k)}_k)$. If $\ell \in \mathcal{C}^{\delta}$, then by Equation \eqref{eq: estimate 1}, 
	\begin{equation*}
		D_{\ell^{(k)}}f^{(k)} = t\left( \begin{array}{ccc}
			e^{t\ell_1^{(k)}} & & \\
			& \ddots &       \\
			&& e^{t\ell_k^{(k)}}
		\end{array}\right)\left(I+A^{(k)}\right),
	\end{equation*}
	where $A^{(k)}$ is a $k\times k$ matrix whose entries are $O(e^{-t\delta})$. Thus, for $t$ sufficiently large, $D_{\ell}f$ is invertible. By the Implicit Function Theorem,
	\begin{equation}\label{eq: implicit function theorem}
		\left(\frac{\partial \ell}{\partial \zeta} \bigg\vert \frac{\partial \ell}{ \partial \varphi}\right) 
		= -\left(D_{\ell}f\right)^{-1}\left(\frac{\partial f}{\partial \zeta} \bigg\vert \frac{\partial f}{ \partial \varphi}\right).
	\end{equation} 
	The function $f^{(k)}$ does not depend on $\zeta^{(k')}$'s for $k'>k$. Therefore, the $(n+m)\times (n+m)$ matrix 
	\[ 
		\left(\frac{\partial f}{\partial \zeta}\right) 
		= \begin{pmatrix}
			\dfrac{\partial f^{(n)}}{\partial \zeta^{(n)}} & \dfrac{\partial f^{(n)}}{\partial \zeta^{(n-1)}} & \cdots & \dfrac{\partial f^{(n)}}{\partial \zeta^{(1)}}\\[0.3cm]
			 & \dfrac{\partial f^{(n-1)}}{\partial \zeta^{(n-1)}} &  \cdots &         \\[0.3cm]
			 &  & \ddots & \\[0.3cm]
			 &  &  & \dfrac{\partial f^{(1)}}{\partial \zeta^{(1)}}
		  \end{pmatrix}
	\]
	is block upper-triangular, as is the $(n+m)\times m$ matrix $\left(\frac{\partial f}{\partial \varphi}\right)$. Since $\zeta\in \mathcal{C}^{\delta}$, by Equation \eqref{eq: estimate 2}, for each $i\leq k$,
	\[
		\left(\frac{\partial f^{(k)}}{\partial \zeta^{(i)}}\right) 
		= -2t\left(\begin{array}{ccc}
			e^{2t\zeta_1^{(k)}} &  &  \\
			 & \ddots &     \\	
			 & & e^{2t\zeta_k^{(k)}}\\
		\end{array}\right)B^{(k,i)},
	\]
	where for $k=i$, $B^{(k,k)} = I + B^{(k)}$ is $k\times k$ with the entries of matrix $B^{(k)}$ in $O(e^{-t\delta})$, and for $k>i$, $B^{(k,i)}$ is a $k\times i$ matrix with entries in $O(e^{-t\delta})$. Similarly, for each $i\leq k$,
	\[
		\left(\frac{\partial f^{(k)}}{\partial \varphi^{(i)}}\right) 
		= -2t\left(\begin{array}{ccc}
			e^{2t\zeta_1^{(k)}} &  &  \\
			& \ddots &     \\	
			& & e^{2t\zeta_k^{(k)}}\\
		\end{array}\right)C^{(k,i)},
	\]
	where $C_t^{(k,i)}$ is a $k\times (i-1)$ matrix with entries in $O(t^{-1}e^{-t\delta})$ (we absorb a factor of $t^{-1}$ into $C_t^{(k,i)}$ to simplify the equation on the next line). Thus,
	\[ 
		\left(\frac{\partial f_i}{\partial \zeta_j} \bigg\vert \frac{\partial f_i}{ \partial \varphi_j}\right)_{i,j} 
		= -2t \left(\begin{array}{ccc}
			e^{2t\zeta_1^{(n)}} & &  \\
			& \ddots &            \\	
			& & e^{2t\zeta_1^{(1)}} \\
		\end{array}\right)\left(B^{(k,i)}\,\big\vert\, C^{(k,i)}\right)_{k,i}
	\]
	(note that the ordering of the diagonal entries is the same as for $\ell$ and $f$ above). Plugging this into Equation \eqref{eq: implicit function theorem}, we have
	\[ \left(\frac{\partial \ell}{\partial \zeta} \bigg\vert \frac{\partial \ell}{ \partial \varphi}\right) 
		= 2\left(I+A\right)^{-1}\left(\begin{array}{ccc}
			e^{t(2\zeta_1^{(n)}-\ell_1^{(n)})} & & \\
			& \ddots &           \\	
			& & e^{t(2\zeta_1^{(1)}-\ell_1^{(1)})} \\
		\end{array}\right)\left(B^{(k,i)}\,\big\vert\, C^{(k,i)}\right)_{k,i}, \] 
	where $A$ is the block diagonal matrix whose diagonal blocks are $A^{(k)}$. The matrix $\left(I+A\right)^{-1} = I + O(e^{-t\delta})$, and the result can now be read from the form of this matrix. 
\end{proof}

In our chart $\mathbb{R}^{n+m}\times T^m$, the  Hamiltonian vector field of $\ell_i^{(k)}(\zeta,\varphi)$ with respect to $\pi_t$ is
\[
	X_{\ell_i^{(k)}} = \left\{\ell_i^{(k)} \circ {\rm gw}_t^{-1}\circ \Delta_t^{-1},- \right\}_t
\]
and the Hamiltonian vector field of $\zeta_i^{(k)}$ with respect to $\pi_{\infty}$ is $X_{\zeta_i^{(k)}} = \{\zeta_i^{(k)},- \}_{\infty}.$

\begin{lemma}\label{lemma: convergence of vector fields}
	For all $1\leq i \leq k \leq n$ and $\delta>0$, $X_{\ell_i^{(k)}}$ converges uniformly to $ 2X_{\zeta_i^{(k)}}$ on $\mathcal{C}^{\delta}\times T^m$.
\end{lemma}

\begin{proof} In matrix notation, the Hamiltonian vector field 
	\[
		X_{\ell_i^{(k)}} =
		\begin{pmatrix}
			\left\{ \zeta_j, \zeta_i\right\}_t   & \left\{ \varphi_j, \zeta_i\right\}_t   \\[0.3cm]
			\left\{ \zeta_j, \varphi_i\right\}_t & \left\{ \varphi_j, \varphi_i\right\}_t \\
		\end{pmatrix}_{i,j}
		\begin{pmatrix}
			\dfrac{\partial\ell_i^{(k)}}{\partial \zeta_j} \\[0.5cm]
			\dfrac{\partial\ell_i^{(k)}}{\partial \varphi_j} \\
		\end{pmatrix}_{j}
	\]
	Combining Theorem \ref{th: AD theorem} and Lemma \ref{lemma: convergence of derivatives}, for $(\zeta,\varphi) \in \mathcal{C}^{\delta}\times T^m$ and $t$ sufficiently large, 
	\begin{equation*}
	    \begin{split}
	        X_{\ell_i^{(k)}} & = 2e^{t(2\zeta_i^{(k)}-\ell_i^{(k)})}\sum_{q,p} \left\{\zeta_i^{(k)},\varphi_q^{(p)} \right\}_{\infty} \frac{\partial}{\partial \varphi_q^{(p)}} + O(e^{-t\delta}) \\
	        & = 2e^{t(2\zeta_i^{(k)}-\ell_i^{(k)})} \left\{\zeta_i^{(k)},- \right\}_{\infty} + O(e^{-t\delta}). \\
	    \end{split}
	\end{equation*}
	The result follows by Proposition \ref{proposition 5.1}, since $e^{t(2\zeta_i^{(k)}-\ell_i^{(k)})}\to 1$ as $t\to \infty$.
\end{proof}

\subsection{Totally positive matrices and fibers of the Flaschka-Ratiu  system}\label{ss2}

Let $AN_{\mathbb{R}}$ be the set of matrices in $AN$ with real entries. The hypersurfaces defined by equations
$\Delta_{i}^{(k)} = 0$
divide $AN_{\mathbb{R}}$ into chambers.  The chamber of ``totally positive'' matrices is
\[
    AN_+ := \left\{ b \in AN_{\mathbb{R}} \colon \, \Delta_i^{(k)}(b)>0,\,  \forall 1 \leq i < k \leq n\right\}.
\]
The restrictions of the functions $\varphi_{i}^{(k)}$ to $AN_{\mathbb{R}}$, defined where $\Delta_{\ell}^{(k)}\neq 0$, take values in $\{0,\pi\}$. Each chamber of $AN_{\mathbb{R}}$ is a joint level set of these functions. By the description above, $AN_+$ is the joint level set where every $\varphi_{i}^{(k)} = 0$.

The set $AN_{\mathbb{R}}$ is also divided into chambers by the Flaschka-Ratiu systems (i.e. by the hypersurfaces where singular values $\ln(\lambda_i^{(k)}(bb^*))$ collide). By Theorem \ref{theorem ginzburg-weinstein map}, each of these chambers equals ${\rm gw}_t(\mathcal{S})$, for a connected component $\mathcal{S}$ of ${\rm Sym}_0(n) = {\rm Sym}(n)\cap \mathcal{H}_0$.  

Although the two chamber structures are different, in Lemma \ref{lemma: connected component} we prove there is a connected component $\mathcal{S} \subseteq {\rm Sym}_0(n)$ so that for arbitrary $\delta>0$ and $t$ sufficiently large, the subset ${\rm gw}_t(\mathcal{S}^{\delta}) \subseteq {\rm gw}_t(\mathcal{S})$ is contained in $AN_+$. Here $\mathcal{S}^{\delta} = \mathcal{S}\cap L^{-1}(\mathcal{C}^{\delta})$. As we will see in the next subsection, this implies there is a choice of angle coordinates $\psi$ for the Gelfand-Zeitlin system such that for sufficiently large $t$, points in the Lagrangian section where $\psi=0$ are sent by the scaled Ginzburg-Weinstein map to points in the Lagrangian section where $\varphi=0$.

\begin{lemma}\label{lemma: correspondence}
	The map $\bm{a}\bm{z}_{0}\colon (\mathbb{R}_+)^{n+m}\times T^m \to AN$ gives a 1-1 correspondence between chambers of $AN_{\mathbb{R}}$ and coordinates $\phi \in \{0,\pi\}^m$, i.e. the chambers of $AN_{\mathbb{R}}$ equal the images $\bm{a}\bm{z}_{0}((\mathbb{R}_+)^{n+m}\times \{\phi\})$ for fixed $\phi \in \{0,\pi\}^m$, and these images are distinct. 
\end{lemma}

This lemma implies that the chambers of $AN_{\mathbb{R}}$ are connected.


\begin{proof}
    We can prescribe the chamber of $AN_{\mathbb{R}}$ containing $\bm{az}_0(x,z) = \bm{a}(x)\bm{z}_0(z)$ by prescribing the $z$'s inductively. We ignore the matrix factor $\bm{a}(x)$, as it is always positive. As shown in Figure \ref{fig 5}, label the coordinates $z_1,\dots,z_m$ by
	\[
		z_{1,2},  z_{1,3}, z_{2,3},\dots, z_{1,n},\dots, z_{n-1,n}.
	\]
	\begin{figure}[h!]
	    \centering
        \[
	    \begin{tikzpicture}[baseline={([yshift=-.5ex]current bounding box.center)}]
		\draw (0.5,0)--(4.5,0);
		\draw (0.5,0.5)--(4.5,0.5);
		\draw (0.5,1)--(4.5,1);
		\draw (0.5,1.5)--(4.5,1.5);
		\draw (2.5,2.1)node{$\vdots$};
		\draw (0.5,2.5)--(4.5,2.5);

		\draw (1.25,1.5)--(1.5,1) node[above right, font=\small]{$z_{1,4}$};
		\draw (2.25,1.5)--(2.5,1) node[above right, font=\small]{$z_{2,4}$};
		\draw (3.25,1.5)--(3.5,1) node[above right, font=\small]{$z_{3,4}$};

		\draw (2,1)--(2.25,0.5) node[above right, font=\small]{$z_{1,3}$};
		\draw (3,1)--(3.25,0.5) node[above right, font=\small]{$z_{2,3}$};

		\draw (2.75,0.5)--(3,0) node[above right, font=\small]{$z_{1,2}$};
		\end{tikzpicture}
        \]
	    \caption{Label of variables.}
	    \label{fig 5}
	\end{figure}
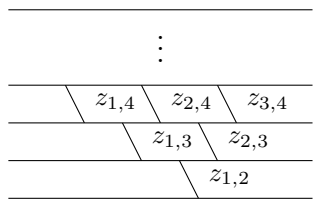
    The minor $\Delta_1^{(2)}(\bm{z}_0(z)) = z_{1,2}$, so we can prescribe its sign by prescribing the sign of $z_{1,2}$. Assume we have prescribed the signs of the minors $\Delta_q^{(p)}(\bm{z}_0(z))$ by prescribing the coordinates $z_{q,p}$,  for all $1\leq q < p<k$. By Lindstr\"om lemma, for all $1\leq i < k$, 
	\begin{equation*}
		\Delta_i^{(k)}(\bm{z}_0(z)) = z_{i,k}\cdot\left( \mbox{ monomial in } z_{q,p},\,  q<k \right).
	\end{equation*}
	Thus, we can prescribe the signs of $\Delta_i^{(k)}(\bm{z}_0(z))$ by prescribing the signs of $z_{i,k}$.
\end{proof}

Combining Lemma \ref{lemma: correspondence} and Proposition \ref{proposition 4.2}, for all $\delta>0$, there exists a $t_0>0$ such that, for all $t\geq t_0$ and $\ell \in \mathcal{C}^{\delta}$, there exists $w\in W^{\delta/2}$ and $\phi \in T^m$, so that 
$$\bm{a}\bm{z}_{0}(tw,\phi) \in \mathcal{L}_t^{-1}(\ell)\cap AN_{\mathbb{R}}$$
is contained whichever chamber of $AN_{\mathbb{R}}$ we please, by choosing $\phi$ appropriately. Thus, the fiber $\mathcal{L}_t^{-1}(\ell)$ intersects every chamber of $AN_{\mathbb{R}}$ at least once.

\begin{lemma}\label{lemma: intersection}
	For all $\delta > 0$, there exists $t_0 >0$ such that for all $t \geq t_0$ and $\ell\in \mathcal{C}^{\delta}$, the fiber $\mathcal{L}_t^{-1}( \ell)$ intersects every chamber of $AN_{\mathbb{R}}$ exactly once.
\end{lemma}

\begin{proof} Fix $\delta>0$ and let $t_0$ as in Proposition \ref{proposition 4.2}. The fiber $L^{-1}(\ell)$ intersects ${\rm Sym}(n)$ at exactly $2^m$ points (this follows directly from linear algebra). By Theorem \ref{theorem ginzburg-weinstein map}, $ \mathcal{L}_t^{-1}(\ell) = {\rm gw}_t(L^{-1}(\ell))$, and ${\rm gw}_t(x) \in AN_{\mathbb{R}}$ if and only if $x \in {\rm Sym}(n)$. Thus, the fibers of $\mathcal{L}_t$ intersect $AN_{\mathbb{R}}$ at exactly $2^m$ points.  Since there are $2^m$ chambers in $AN_{\mathbb{R}}$, and $\mathcal{L}_t^{-1}(\ell)$ intersects every chamber of $AN_{\mathbb{R}}$ at least once, this completes the proof.
\end{proof} 

Next, we show that there is a unique connected component $\mathcal{S} \subseteq {\rm Sym}_0(n)$ whose elements are sent to $AN_+$ by ${\rm gw}_t$ for $t$ sufficiently large. For a connected component $\mathcal{S} \subseteq {\rm Sym}_0(n)$, let $\mathcal{S}^{\delta} := \mathcal{S}\cap L^{-1}(\mathcal{C}^{\delta})$.

\begin{lemma}\label{lemma: connected component}
	There is a unique connected component $\mathcal{S} \subseteq {\rm Sym}_0(n)$ such that for all $\delta>0$, there exists a $t_0\geq 0$ such that for all $t  \geq t_0$ and $x\in \mathcal{S}^{\delta}$,
	\[
	    {\rm gw}_t(x)\in AN_+.
	\]
\end{lemma}

\begin{proof} Fix $\delta>0$ arbitrary. By Lemma \ref{lemma: intersection}, there exists $t_0\geq 0$ such that for all $t\geq t_0$ and $\ell\in \mathcal{C}^{\delta}$, $\mathcal{L}_t^{-1}(\ell)$ intersects every chamber of $AN_{\mathbb{R}}$ exactly once. 

Consider the set 
\[
    {\rm gw}_{t_0}^{-1}(AN_+)\cap {\rm Sym}^{\delta}(n),
\]
where ${\rm Sym}^{\delta}(n) := {\rm Sym}(n)\cap L^{-1}(\mathcal{C}^{\delta})$ (recall $L$ is the Gelfand-Zeitlin map with coordinates $\ell_i^{(k)}$).  By Lemma \ref{lemma: intersection}, the image of this set under $L$ equals $\mathcal{C}^{\delta}$. 

If there exists disjoint open subsets $A,B\subseteq {\rm Sym}^{\delta}(n)$, so that 
\[
    {\rm gw}_{t_0}^{-1}(AN_+)\cap {\rm Sym}^{\delta}(n) = A\cup B,
\]
then 
\[
    \mathcal{C}^{\delta} = L(A\cup B) = L(A) \cup L(B).
\]
Since the restriction of $L$ to $\mathcal{H}_0$ is open as a map to $\mathcal{C}^{\circ}$, $L(A)$ and $L(B)$ are open. Since for all $\ell \in \mathcal{C}^{\delta}$, $\mathcal{L}_{t_0}^{-1}(\ell)$ intersects $AN_{+}$ exactly once, the sets $L(A)$ and $L(B)$ are disjoint. This implies that $\mathcal{C}^{\delta}$ is not connected, which is a contradiction (it is convex). Thus, the set ${\rm gw}_{t_0}^{-1}(AN_+)\cap {\rm Sym}^{\delta}(n)$ is connected, and must be contained in a connected component $\mathcal{S}\subseteq {\rm Sym}_0(n)$.

The result follows for all $t\geq t_0$, since ${\rm gw}_{t_0}^{-1}(AN_+)\cap {\rm Sym}^{\delta}(n) = \mathcal{S}^{\delta}$ and $t\mathcal{S}^{\delta} \subseteq \mathcal{S}^{\delta}$ for all $t\geq 1$. Explicitly, given $t\geq t_0$ and $x\in \mathcal{S}^{\delta}$, we have by definition of ${\rm gw}$ that
\[
    {\rm gw}_{t}(x) = {\rm gw}_{t_0}\left(\frac{t}{t_0}x\right) \in AN_+. \qedhere
\]
\end{proof}

\subsection{Proof of Proposition \ref{proposition 6.1} and Theorem \ref{main theorem}}\label{ss3}

\begin{proof}[Proof of Proposition \ref{proposition 6.1}]
Fix an arbitrary regular fiber of the Gelfand-Zeitlin system, $L^{-1}(\ell_0)$. Since $\ell_0\in \mathcal{C}$,  $\ell_0\in \mathcal{C}^{\delta}$ for some $\delta >0$. We will show the functions $\varphi_i^{(k)}\circ \Delta_t\circ {\rm gw}_t$ converge on $L^{-1}(\ell_0)$ by showing that they converge at a point in $L^{-1}(\ell_0)$ and their derivatives converge uniformly on $L^{-1}(\ell_0)$.

Let $\mathcal{S}\subseteq {\rm Sym}_0(n)$ be the unique connected component described in Lemma \ref{lemma: connected component}. Choose angle coordinates $\psi$ for the Gelfand-Zeitlin system so that $\psi\vert_{\mathcal{S}} = 0$. Let $x$ be the unique point in $ \mathcal{S}\cap L^{-1}(\ell_0)$ that has coordinates $\psi(x) = 0$. By Lemma \ref{lemma: connected component}, for all $t$ sufficiently large, ${\rm gw}_t(x) \in AN_+$. Thus, for all $t$ sufficiently large and $1 \leq q< p\leq n$, 
\[
	\varphi_q^{(p)} \circ \Delta_t\circ {\rm gw}_t(x) = 0,
\]
which equals the value of the linear combination $\psi_q^{(p-1)} + \mbox{ ``higher terms''}$ at $x$ for our choice of angle coordinates.

Second, for all $A \in L^{-1}(\ell_0)$, $1\leq i\leq k$, and for $t$ sufficiently large, 
\begin{equation}
	\begin{split}
		\sum_{j=1}^i\frac{\partial}{\partial \psi_j^{(k)}}\left( \varphi_q^{(p)}\circ\Delta_t\circ {\rm gw}_t\right)(A) 
		& = \left\{\ell_i^{(k)},\varphi_q^{(p)}\circ\Delta_t\circ {\rm gw}_t \right\}_{\mathfrak{k^*}}(A) \\
		& = \left\{\ell_i^{(k)}\circ {\rm gw}_t^{-1}\circ\Delta_t^{-1},\varphi_q^{(p)}\right\}_t   (\Delta_t\circ {\rm gw}_t(A)). \\
	\end{split}
\end{equation}
By Lemma \ref{lemma: convergence of vector fields},  this converges uniformly on any $\mathcal{C}^{\delta}\times T^m$ to the constant function 
\begin{equation}
	\begin{split}
		2\left\{\zeta_i^{(k)},\varphi_q^{(p)}\right\}_{\infty} 
		& = \left\{\ell_i^{(k)}, \psi_q^{(p-1)} + \mbox{ ``higher terms''}  \right\}_{\mathfrak{k}^*} \\
		& = \sum_{j=1}^i \frac{\partial}{\partial \psi_j^{(k)}} \left(\psi_q^{(p-1)} + \mbox{ ``higher terms''}\right).
	\end{split}
\end{equation}
Since $\Delta_t\circ {\rm gw}_t(A)$ is contained in $\mathcal{C}^{\delta/2}\times T^m$ for $t$ sufficiently large, this completes the proof. 
\end{proof}

\begin{proof}[Proof of Theorem \ref{main theorem}] As observed in the proof of \cite[Theorem 7]{AD}, there is a unique Poisson isomorphism from $\mathcal{H}_0$ equipped with $\pi_{\mathfrak{k}^*}$ to $\mathcal{C}\times T^m$ equipped with $\pi_{\infty}$ such that 
\[
    \zeta_i^{(k)} = \frac{1}{2} \ell_i^{(k)}
\]
and 
\[
    \varphi_i^{(k)} = \psi_i^{(k-1)} + \mbox{linear combination of higher $\psi$ in the order}.
\]
Combining Propositions \ref{proposition 5.1} and \ref{proposition 6.1}, we see that in coordinates $\zeta,\varphi$ on $\mathbb{R}^{n+m}\times T^m$ the map \eqref{finite t} converges on $\mathcal{H}_0$ to this Poisson isomorphism. This is the Gelfand-Zeitlin system, up to a linear change of coordinates.
\end{proof}

\section{Conclusion}

Recall from Section \ref{section 4} that the tropicalization of a positive polynomial $p(x)$, in complex variables, is equal, on an open dense set, to the ``tropical'' limit
\[
    p^{\mathbb{T}}(w) = \lim_{t\to \infty} \frac{1}{t}\ln \left(p(e^{tw+\sqrt{-1}\phi})\right),
\]
where we have substituted $x=e^{tw+\sqrt{-1}\phi}$. The scaling limit of Ginzburg-Weinstein diffeomorphisms studied in this paper may be viewed as a non-abelian ``tropical limit'': Ginzburg-Weinstein diffeomorphisms can be written as the composition
\[
    {\rm gw}\colon \mathfrak{k}^* \xrightarrow{f} \mathfrak{k}^*\cong i\mathfrak{k} \xrightarrow{\exp} P \xrightarrow{\cong} AN,
\]
where $P = \exp(i\mathfrak{k})$, and $f(A) = Ad^*_{\Psi(A)}A$, $\Psi:\mathfrak{k}^* \to K$, is the flow of a Moser vector field on $\mathfrak{k}^*$ \cite{A,AM}. Using this factorization, we can rewrite the limit of Proposition \ref{proposition 5.1} in the more suggestive form,
\[
    \lim_{t\to \infty} \zeta\circ \Delta_t \circ {\rm gw}_t(A)
    = \lim_{t\to \infty} \frac{1}{t}\ln\left(p(e^{f(tA)})\right),
\]
where $p = p_i^{(k)}$ is a positive polynomial in matrix factorization coordinates, and
we have suppressed the diffeomorphism $P\cong AN$. 

We hope that a more general theory of non-abelian tropical limits including Ginzburg-Weinstein maps for compact Lie groups other than $U(n)$ will emerge. 

{\bf Funding:} Our work was supported in part by the project MODFLAT of the European Research Council (ERC), by the grants number 178794 and number 178828 of the Swiss National Science Foundation (SNSF) and by the NCCR SwissMAP of the SNSF.

{\bf Acknowledgment:} The authors would like to thank B. Hoffman, M. Podkopaeva and A. Szenes for interesting and productive discussions.


\end{document}